\documentclass[10pt, a4paper]{article}
\usepackage{amsmath, amsfonts}
\usepackage[usenames]{color}
\usepackage{extarrows}
\usepackage{multirow}
\usepackage{mathrsfs}
\usepackage[all]{xy}
\usepackage[utf8]{inputenc}
\usepackage[T1]{fontenc}
\usepackage{hyperref} 
\hypersetup{bookmarks = true, 
                  colorlinks = true, 
                  linkcolor = blue, 
                  citecolor = blue, 
                  }
\usepackage{graphicx}
\usepackage{amssymb}
\usepackage{amsmath}
\usepackage{amsthm}

\newcommand{\C}{\mathbb{C}}

\newtheorem{theorem}{Theorem}[section]

\newtheorem{corollary}[theorem]{Corollary}

\newtheorem{definition}[theorem]{Definition}
\newtheorem{example}[theorem]{Example}

\newtheorem{lemma}[theorem]{Lemma}

\newtheorem{proposition}[theorem]{Proposition}
\newtheorem{remark}[theorem]{Remark}

\newcommand{\xdownarrow}[1]{%
  {\left\downarrow\vbox to #1{}\right.\kern-\nulldelimiterspace}
}

\title{\textbf{On tangency in equisingular families of curves and surfaces}}
\author{ \ \ \ \\{A. Giles Flores\footnote{A. Giles Flores: Universidad Aut\'onoma de Aguascalientes, Departamento de Matemáticas y Física, Aguascalientes, México. \hspace{0.5cm} e-mail: arturo.giles@cimat.mx} $\ \ $, $\ \ $ O.N. Silva\footnote{O.N. Silva: Instituto de Matemáticas, Universidad Nacional Aut\'onoma de México (UNAM), Cuernavaca, México. e-mail: otoniel@im.unam.mx}} $\ \ $ and $\ \ $ J. Snoussi\footnote{J. Snoussi: Instituto de Matemáticas, Universidad Nacional Aut\'onoma de México (UNAM), Cuernavaca, México. \hspace{5cm} e-mail: jsnoussi@im.unam.mx}}
\date{}

\begin{document}

\maketitle

\begin{abstract}
We study the behavior of limits of tangents in topologically equivalent spaces. In the context of families of generically reduced curves, we introduce the $s$-invariant of a curve and we show that in a Whitney equisingular family with the property that the $s$-invariant is constant along the parameter space, the number of tangents of each curve of the family is constant. In the context of families of isolated surface singularities, we show through examples that Whitney equisingularity is not sufficient to ensure that the tangent cones of the family are homeomorphic. We explain how the existence of exceptional tangents is preserved by Whitney equisingularity but their number can change.
\end{abstract}

\section{Introduction}

$ \ \ \  $ We say that two germs of analytic sets $(W,0)$ and $(W^{'},0)$ in $(\mathbb{C}^n,0)$ have the same embedded topological type or are topologically equivalent if there exists a germ of homeomorphism $\varphi : (\mathbb{C}^n,W,0) \rightarrow (\mathbb{C}^n,W^{'},0)$.

In 1971, Zariski asked the following question (see \cite[Problem B]{zariski}): if $(W,0)$ and $(W^{'},0)$ are topologically equivalent hypersurfaces, are the projectivized tangent cones of $W$ and $W^{'}$ homeomorphic? In 2005, Fern\'andez de Bobadilla gave a negative answer to Zariski's question by presenting a counter-example (see \cite[Th.A (iii)]{bobadilla}). So, motivated by this question, we are interested in understanding the behavior of limits of tangents in topologically equivalent spaces. In this work we deal with that question ``in family'', first for generically reduced curves and then for families of isolated surface singularities.

In the case of germs of curves in $(\mathbb{C}^n,0)$, we consider a topologically trivial (flat) family of generically reduced curves $p : (X,0) \subset (\mathbb{C}^{n+1},0) \rightarrow (\mathbb{C},0)$  with a section $\sigma: (\mathbb{C},0)\rightarrow (X,0)$ such that the image of $\sigma$ is smooth and the fibers $X_t:=p^{-1}(t)$ have their unique possibly singular point at $\sigma(t)$. In this case, the special curve $(X_0,\sigma(0))$ and the generic curve $(X_t,\sigma(t))$ are topologically equivalent, for all $t$ (see \cite[Th. $9.3$]{greuel}). Thus we can ask under what conditions the topological type of the tangent cone of $(X_t,\sigma(t))$ remains constant. More precisely, we are interested in the following question:

\begin{flushleft}
\textbf{Question 1:} \textit{ If $p:(X,0)\rightarrow (\mathbb{C},0)$ is a topologically trivial flat family of generically reduced curves, under what conditions are the Zariski tangent cones $C(X_0,\sigma(0))$ and $C(X_t,\sigma(t))$ homeomorphic?}
\end{flushleft}

Since set-theoretically the tangent cone of a curve is a finite union of lines, then the cones $C(X_0,\sigma(0))$ and $C(X_t,\sigma(t))$ are homeomorphic if and only if the number of tangents of $(X_t,\sigma(t))$ is constant for all $t$. Therefore, the above question can be also reformulated in terms of numbers of tangents of $(X_t,\sigma(t))$.

We can see that Whitney equisingularity is not a sufficient condition for Question 1; in fact, we present an example of a Whitney equisingular family of curves where the number of tangents is not constant (see Example \ref{exemploprinc2}). 

Our approach to answer Question 1 is to study the problem of the number of tangents under generic linear projections of the fibers $(X_t,\sigma(t))$ to $(\mathbb{C}^2,0)$. If $(\mathscr{C} ,0) \subset (\mathbb{C}^n,0) $ is a germ of reduced curve and $\pi:(\mathscr{C},0)\rightarrow (\mathbb{C}^2,0)$ is a generic projection in the sense of (\cite[Ch. IV]{briancon}), then the image $(\pi(\mathscr{C}),0)$ of $(\mathscr{C},0)$ by $\pi$ is a reduced curve with the structure given by the $0$-Fitting ideal of $\pi_{\ast}(\mathcal{O}_{\mathscr{C},0})$ (see \cite[Sec. $1$]{Teissier8}). Furthermore, the number of tangents of $(\mathscr{C},0)$ and $(\pi(\mathscr{C}),0)$ is the same. Thus, for a germ of curve $(\mathscr{C} ,0) \subset (\mathbb{C}^n,0) $ with two or more branches we define in Section \ref{sections} the $s$-invariant as the sum of the intersection multiplicities of the branches of its generic projection $(\pi(\mathscr{C}),0)$. 

In order to make use of this invariant, we need to establish its semi-continuity in certain types of deformations, so  we propose to find a common generic projection for $(X_t,\sigma(t))$ for all $t$. Thus, another question is: 

\begin{flushleft}
\textbf{Question 2:} \textit{ If $p:(X,0) \rightarrow (\mathbb{C},0)$ is a flat family of generically reduced curves, under what condition does there exist a common generic projection for $(X_t,\sigma(t))$ for all $t$?}
\end{flushleft}

Briançon, Galligo and Granger proved (\cite[Th. IV.$8$]{briancon}) that if $p:(X,0) \rightarrow (\mathbb{C},0)$ is an equisaturated family of reduced curves then there exists a common generic projection for $X_t$ for all $t \in T$. We give an answer to Question 2 by  extending this result without any  equisingularity hypothesis on $(X,0)$ (see Theorem \ref{thgenericproj}).

Using the above result, we finally give an answer to Question 1. More precisely, we show that in a topologically trivial family, if the multiplicity and the $s$-invariant of the fibers are constant, then their tangent cones are homeomorphic (Corollary \ref{corprinc}). As a consequence of our result, we show that if $p:(X,0)\rightarrow (\mathbb{C},0)$ is an equisaturated family of reduced curves, then the cones $C(X_0,\sigma(0))$ and $C(X_t,\sigma(t))$ are homeomorphic (Corollary \ref{cor1}); this result is a particular case of a result by Sampaio (see \cite[Th. $2.2$]{edson}). We also present in Section \ref{sec5} counterexamples to some natural questions about the results.

In the last part of this work, we deal with the case of families of surfaces with isolated singularities. We present an example showing that, in a Whitney equisingular family of isolated surfaces singularities, the tangent cones of the elements of the family need not be homeomorphic (see Example \ref{exewhi}).

However, in surfaces we have a new phenomenon related to limits of tangents: the (possible) existence of exceptional tangents (see \cite{LT88}). In Section \ref{sec4}, we first explain how the existence of exceptional tangents is preserved in Whitney regular families; this fact is a direct consequence of results by Lê and Teissier (\cite[Prop. 2.2.4.2 and Thm. 2.3.2]{LT88}), and seems to be known to specialists. Then we consider the following question:

\begin{flushleft}
\textbf{Question 3:} \textit{ In a Whitney equisingular family of isolated surface singularities, is the number of exceptional tangents constant?}
\end{flushleft}

The answer is negative and the example we present (Example \ref{exefinal}) is one of the classical examples of Briançon and Speder studied in their paper \cite{BS75a} (see also \cite[Intro.]{pichon2}). We thank Anne Pichon for pointing it out to us.
 
Throughout the paper, $(x_1,\cdots,x_n,t)$ denotes a local coordinate system for $\mathbb{C}^n \times \mathbb{C}$. The ring $\mathcal{O}_n\simeq \mathbb{C}\lbrace x_1,\cdots,x_n \rbrace$ (respectively $\mathcal{O}_{n+1}\simeq \mathbb{C}\lbrace x_1,\cdots,x_n,t \rbrace$) denotes the local ring of holomorphic functions on $(\mathbb{C}^n,0)$ (respectively on $(\mathbb{C}^n \times \mathbb{C},0)$); that is, we use the last coordinate in $\mathbb{C}^n \times \mathbb{C}$ for the variable $t$ (with the exception of the proof of Theorem \ref{thgenericproj}).
 
\section{Preliminary Results}\label{sec2}

\subsection{Families of curves}

\begin{definition}\label{defsur} (a) Let $(\mathscr{C},0)$ be a germ of curve in $(\mathbb{C}^{n},0)$. We say that $(\mathscr{C},0)$ is generically reduced if the only possibly non reduced point of the curve near the origin is the origin itself, that is, a curve with an isolated singularity at the origin.\\ 

\noindent (b) Let $(X,0) \subset (\mathbb{C}^{n+1},0) $ be a germ of a reduced and pure dimensional complex surface. Consider a projection $p:(X,0)\rightarrow (\mathbb{C},0)$. We say that the surface $(X,0)$ is a one-parameter flat deformation of the germ of reduced curve $(X_0,0):=(p^{-1}(0),0)$ if the projection $p$ is a flat map, which is equivalent, in this setting, to say that $p\in {\mathcal O}_{X,0}$ is neither a zero divisor nor a unit. When $(X_0,0)$ is generically reduced, we say $p:(X,0)\rightarrow (\mathbb{C},0)$, or simply $(X,0)$, is a family of generically reduced curves.\\

\noindent (c) Given a representative $p:X\rightarrow D$ of a family of generically reduced curves $(X,0)$, we will denote the fibers of $p$ by $X_t:=p^{-1}(t)$.

\end{definition}

As illustrated in Example \ref{exemplo2.1}, when $(X,0)$ is not Cohen-Macaulay the special fiber $(X_0,0)$ will always have an embedded component by \cite[Corollary 6.5.5]{jong} (see also \cite[Theorem 17.3]{matsumura}).

\begin{example}\label{exemplo2.1} Let $(X,0)$ be the germ of surface parametrized by the map 

$$\begin{array}{rcl}
n:(\mathbb{C}^2,0) & \rightarrow & (X,0) \subset (\mathbb{C}^4,0),\\
(u,t) & \mapsto & (tu, u^2, u^3, t)\\
\end{array}
$$

\noindent it is a reduced surface with an isolated singularity at the origin, defined by the ideal 

$$ \mathcal{I} =\langle x^2-t^2y,xy-tz,xz-ty^2,z^2-y^3\rangle \mathbb{C} \lbrace x,y,z,t \rbrace.$$

The restriction to $(X,0)$ of the canonical projection $p:(\mathbb{C}^4,0)\rightarrow (\mathbb{C},0)$ to the last factor makes the surface into a one parameter deformation of the curve $(X_0,0)$ defined by the ideal: 

$$ \mathcal{I}_{0}=\langle x^2,xy,xz,z^2-y^3\rangle \mathbb{C} \lbrace x,y,z \rbrace.$$

One can see that the curve has an embedded component at the origin and that the surface $(X,0)$ is not Cohen-Macaulay (see also Example \rm\ref{example2.2}\textit{).}
\end{example}

\subsection{The invariants}

$ \ \ \  $ Following Brücker and Greuel \cite[p. 96]{greuel2}, we recall the definitions of the $\delta$-invariant and the Milnor number for a generically reduced curve.

\begin{definition}
Let $(\mathscr{C},0) \subset (\mathbb{C}^{n},0)$ be a germ of a generically reduced curve defined by the ideal $ \mathcal{I}_{(\mathscr{C},0)}\subset {\mathcal O}_n$. Denote by $(|\mathscr{C}|,0)$ the corresponding curve with reduced structure, and let $n:(\overline{\mathscr{C}},\overline{0}))\rightarrow (|\mathscr{C}|,0)$ be the normalization of $(|\mathscr{C}|,0)$. The following numbers:

\begin{center}
$\delta(|\mathscr{C}|,0):=dim_{\mathbb{C}} \left(\dfrac{\mathcal{O}_{(\overline{\mathscr{C}},\overline{0})}}{\mathcal{O}_{(|\mathscr{C}|,0)}} \right)$, $ \ \ \ $ $ \epsilon(\mathscr{C},0):=dim_{\mathbb{C}} \dfrac{\sqrt{\mathcal{I}_{(\mathscr{C},0)}\mathcal{O}_{(\mathscr{C},0)}}}{\mathcal{I}_{(\mathscr{C},0)}\mathcal{O}_{(\mathscr{C},0)}}$ 
$\ \ \ \ $ and $ \ \ \ \ $ $\delta(\mathscr{C},0):=\delta(|\mathscr{C}|,0)-\epsilon(\mathscr{C},0)$\\
\end{center}

\noindent are respectively the delta-invariant of $(|\mathscr{C}|,0)$, the epsilon-invariant and the delta-invariant of $(\mathscr{C},0)$. Furthermore,

\begin{center}
$m(\mathscr{C},0)$, $ \ \ \ \ $ $\mu(|\mathscr{C}|,0)=2\delta(|\mathscr{C}|,0)-r(\mathscr{C},0)+1 \ \ \ \ $ and $\ \ \ \ \mu(\mathscr{C},0)=\mu(|\mathscr{C}|,0)-2\cdot \epsilon(\mathscr{C},0)$
\end{center}

\noindent are respectively the Hilbert-Samuel multiplicity of the maximal ideal $M(\mathcal{O}_{(\mathscr{C},0)})$ of the local ring $\mathcal{O}_{(\mathscr{C},0)}$ of $\mathscr{C}$ at $0$ and the Milnor number of $(|\mathscr{C}|,0)$ and $(\mathscr{C},0)$; where $r(\mathscr{C},0)$ is the number of branches of $(\mathscr{C},0)$.

\end{definition}

\begin{example}\label{example2.2} 
Looking again at Example \rm \ref{exemplo2.1}, \textit{a primary decomposition of the ideal $\mathcal{I}_0$ in $\mathcal{O}_3 \simeq \mathbb{C}\lbrace x,y,z \rbrace$ is given by}

$$\mathcal{I}_{0}=\langle x^2,xy,xz,z^2-y^3\rangle\mathcal{O}_3=\langle \langle x,z^2-y^3\rangle\cap \langle x^2,xy,y^3,z \rangle\rangle\mathcal{O}_3.$$

\noindent \textit{The local ring of the corresponding reduced curve is $\dfrac{{\mathcal{O}_3}}{\langle x, z^2 - y^3\rangle \mathcal{O}_3}$, so $\delta(|X_{0}|,0)=1$ and $\mu(|X_0|,0)=2$. Also, by} \rm\cite[\textit{Corollary 5.7}]{greuel} \textit{we have that}  

$$\epsilon(X_{0},0)=dim_{\mathbb{C}}\dfrac{\mathbb{C}\lbrace x,y,z \rbrace}{\langle x^2,xy,y^3,z \rangle}-dim_{\mathbb{C}}\dfrac{\mathbb{C}\lbrace x,y,z \rbrace}{\langle x^2,xy,y^3,z,x,z^2-y^3 \rangle}=4-3=1$$.

\textit{Therefore, $\delta(X_{0},0)=0$, $\mu(X_{0},0)=0$ and another calculation shows that $m(X_0,0)=2$.}

\end{example}

\begin{remark}
\textit{The multiplicity of a germ of generically reduced curve $(\mathscr{C},0)$ is the same as the one of the corresponding reduced curve $(|\mathscr{C}|,0)$} (see \rm \cite[Lemma $4.8$]{otoniel2}\textit{).}
\end{remark}

\subsection{Topological triviality, Whitney equisingularity and Equisaturation}\label{sectiontop}

$ \ \ \  $ We will recall here the concepts of topological triviality and Whitney equisingularity, along a section, and state the main criterion we will be using. Given a family of generically reduced curves $p:(X,0)\rightarrow (\mathbb{C},0)$, we will consider a good representative $p:X\rightarrow T$ in the sense of \cite[p. 248]{buch} (see also \cite[Appendix]{greuel} or \cite[Section 2]{otoniel}).

\begin{definition}\label{deftoptrivial} Let $p:(X,0) \subset (\mathbb{C}^{n+1},0) \rightarrow (\mathbb{C},0)$ be a family of generically reduced curves and suppose that there is a good representative $p:X\rightarrow T$ with a section $\sigma:T \rightarrow X$, such that both $\sigma(T)$ and $X_t \setminus \sigma(t)$ are smooth for $t \in T$.\\ 

\noindent \textit{(a) We say that $p: X\rightarrow T$ is topologically trivial (or for simplicity, $X$ is topologically trivial) if there is a homeomorphism $h: X \rightarrow X_{0} \times T$ such that $p=p^{'} \circ h$, where $p^{'} :X_{0} \times T \rightarrow T$ is the projection on the second factor}.\\

\noindent \textit{(b) Let $p:X \rightarrow T$ be a topologically trivial family of generically reduced curves. If in addition, $h$ is bi-Lipschitz (that is, $h$ and its inverse $h^{-1}$ are Lipschitz maps with the ambient metric), then we say that $X$ is bi-Lipschitz equisingular.}\\

\noindent \textit{(c) We say that $p:X\rightarrow T$ is Whitney equisingular (or for simplicity, $X$ is Whitney equisingular) if the stratification} $ \lbrace X \setminus \sigma(T), \sigma(T) \rbrace $ \textit{satisfies  Whitney's conditions (a) and (b) at $0$, that is:}\\

\textit{For any sequences of points $(x_n)\subset X \setminus \sigma(T)$ and $(t_n) \subset \sigma(T) \setminus \lbrace 0 \rbrace$ both converging to $0$ and such that the sequence of lines $(x_nt_n)$ converges to a line $l$ and the sequence of directions of tangent spaces, $T_{x_n}X$, to $X$ at $x_n$, converges to a linear space $H$ we have:}

\begin{flushleft}
\textit{(Whitney's condition a): The tangent space to $\sigma(T)$ at $0$ is contained in $H$.}\\
\textit{(Whitney's condition b): The line $l$ is contained in $H$.}
\end{flushleft}

\noindent \textit{(d) If $(X,0)$ is a family of reduced curves, we say that $p:X\rightarrow T$ is equisaturated (or for simplicity $X$ is equisaturated) if the Lipschitz saturation of the local rings of $(X_0,0)$ and $(X_t,\sigma(t))$ are isomorphic} \rm(\textit{see for example} \rm\cite[\textit{p.36}]{briancon}).

\end{definition}

\begin{theorem}\label{teotoptrivial} Let $p:(X,0)\rightarrow (\mathbb{C},0)$ be a family of generically reduced curves, with $(X,0)$ reduced and equidimensional. Suppose there is a good representative $p:X\rightarrow T$ with a section $\sigma:T \rightarrow X$, such that both $\sigma(T)$ and $X_t \setminus \sigma(t)$ are smooth for $t \in T$. Then: \\

\noindent \textit{(a) $X$ is topologically trivial, if and only if, both $\delta(X_{t}, \sigma(t))$ and $r(X_{t},\sigma(t))$ are constant for all $t\in T$, if and only if, $\mu(X_{t},\sigma(t))$ is constant and $X_t$ is connected for all $t\in T$.}\\

\noindent \textit{(b) $X$ is Whitney equisingular if and only if both $\mu(X_{t},\sigma(t))$ and $m(X_t, \sigma(t))$ are constant for all $t\in T$.}\\

\noindent \textit{(c) When $X$ is a family of reduced curves, $X$ is equisaturated if and only if it is topologically trivial and the topological type of a generic projection of $X_t$ to ${\mathbb C}^2$ (see Def.} \rm\ref{defgenericproj}\textit{) does not depend on $t\in T$.}

\end{theorem}

Statement $(a)$ is proved in \cite[Theorem 9.3]{greuel}. Statement $(b)$ is proved in \cite[Theorem $3.54$]{otoniel2}, see also \cite[Theorem 4.7]{otoniel}; $(c)$ is proved \cite[Théorème IV.8]{briancon}. In our setting, we always ask that $X$ has pure dimension $2$. We remark that in \cite{greuel}, Greuel considers a more general setting where $X$ is  not necessarily equidimensional. This explains the appearance of the hypothesis on the equidimensionality of $(X,0)$ in Theorem \ref{teotoptrivial}. 

\begin{example}\label{exemplo2.3} Consider again the family of curves of Example \rm\ref{exemplo2.1}. \textit{Take a good representative $p:X\rightarrow T$ and consider the section $\sigma: T\rightarrow X$ defined by $\sigma(t)=(0,0,0,t)$. Note that $X_t \setminus \sigma(t)$ is smooth for all $t\in T$. We have that $\mu(X_t,\sigma(t))=0$ for all $t\in T$. By} \rm\cite[\textit{Proposition 8.8}]{greuel} \textit{we have that $X_t$ is connected for all $t\in T$. Then by Theorem} \rm\ref{teotoptrivial}\textit{(a) $X$ is topologically trivial. Note also that $m(X_0,0)=2$ and $m(X_t,\sigma(t))=1$ for $t\neq 0$, hence $(X,0)$ is not Whitney equisingular by Theorem} \rm\ref{teotoptrivial}\textit{(b)}.
\end{example}

\section{The existence of a common generic projection}\label{secproj}

$ \ \ \  $ The purpose of this section is to give an answer to Question 2, stated in the introduction. Let us precise what we mean by generic projection for a germ of curve.

\begin{definition}\label{defgenericproj} (a) Let $W$ be a representative of a germ of analytic space $(W,0) \subset (\mathbb{C}^n,0)$. We say that a vector $v \in C_5(W,0)$ if there are sequences of points $x_n,y_n \in W$ and numbers $\lambda_n \in \mathbb{C}$ such that $x_n\rightarrow 0$, $y_n\rightarrow 0$ and $\lambda_n\overline{(x_n-y_n)}\rightarrow v$ as $n\rightarrow \infty$ (see \rm\cite[\textit{Section $3$}]{Whi652}\textit{).}\\ 

\noindent \textit{(b) Let $(\mathscr{C},0)$ be a germ of generically reduced curve in $(\mathbb{C}^n,0)$ and $\pi:(\mathbb{C}^n,0)\rightarrow (\mathbb{C}^2,0)$ be a linear projection. We say that the restriction $\pi|_{(\mathscr{C},0)}: (\mathscr{C},0)\rightarrow (\mathbb{C}^2,0)$ is a $C_5$-generic projection for the germ of curve $(\mathscr{C},0)$ if the kernel of $\pi$ intersects $C_5(\mathscr{C},0)$ transversally, that is, $ker(\pi)\cap C_5(\mathscr{C},0)=(0, \cdots, 0)$ (see} \rm\cite[\textit{Chap. IV}]{briancon}\textit{).}\\ 

\noindent \textit{(c) When $\pi$ is $C_5$-generic for $(\mathscr{C},0)$, we will denote the image curve by $( \tilde{{\mathscr{C}}},0)$ and call it a $C_5$-generic projection of $(\mathscr{C},0)$.}\\

\noindent \textit{(d) Let $G(2,n)$ be the Grassmannian of all complex $(n-2)$-dimensional linear subspaces of $\mathbb{C}^n$. For each $H \in G(2,n)$, denote by $\pi_H$ the linear projection parallel to $H$ on a plane $H^{\perp}\simeq \mathbb{C}^2$ (orthogonal to $H$).}
\end{definition}

\begin{remark} Briançon, Galligo and Granger showed that if $(\mathscr{C},0)$ is a germ of analytic reduced singular curve in $(\mathbb{C}^n,0)$, then $C_5(\mathscr{C},0)$ is a finite union of planes, each of these planes containing at least one tangent of the germ of curve $(\mathscr{C},0)$ (see \rm\cite{briancon}, \textit{Theorem $IV.1$). In} \rm\cite{otoniel5}, \textit{we describe a method to determine the $C_5$-cone of a curve $(\mathscr{C},0)$ using a parametrization of $(\mathscr{C},0)$. Notice that if $(\mathscr{C},0)$ is a germ of generically reduced curve, then $C_5(\mathscr{C},0)$ and $C_5(|\mathscr{C}|,0)$ are set theorically the same.}
\end{remark}

In \cite[Prop. IV.8]{briancon}, it is shown that if $p:X \rightarrow T$ is an equisaturated family of reduced curves then there is a Zariski's open (and dense) set $\Omega \subset G(2,n)$ such that for each $H \in \Omega$, $\pi_H$ is a generic projection for $(X_t,\sigma(t))$ for all $t\in T$. Now, we extend this result without any equisingularity hypothesis on $X$.

\begin{lemma} \label{C5asblowup}Let $W$ be a representative of a germ $(W,0) \subset (\C^n,0)$ of analytic set, then there exists an analytic variety $C_5(W)$ with  an analytic map  \[\psi: C_5(W) \to W\]
 such that for every point $q \in W$ we 
have that $\psi^{-1}(q)= C_5(W,q)$.
\end{lemma}

\begin{proof}
      Consider the analytic space $W\times W$ and its blow-up along the diagonal
     \[e_\Delta: Bl_\Delta (W \times W)  \to W\times W\] 
     If we chose coordinates $(z_1,\ldots,z_n,w_1,\ldots,w_n)$ of the ambient space $\C^{2n}$, then we can obtain the space $Bl_\Delta (W \times W)$ as the closure of the graph of the secant map defined away from the diagonal $\Delta$ by:
     \begin{align*} W \times W \setminus \Delta &\longrightarrow \mathbb{P}^{n-1} \\
                            (z,w) & \longmapsto [z_1-w_1: \cdots: z_n-w_n]  \end{align*}
     So we have $Bl_\Delta (W \times W) $ as a subspace of $W \times W \times \mathbb{P}^{n-1}$,  the map $e_\Delta$ is induced by the projection, 
     and the exceptional fiber is the divisor $ D:=e_\Delta^{-1}(\Delta) \subset \Delta \times \mathbb{P}^{n-1}$  which comes with a map $D \to \Delta$ such that for every point 
     $(q,q) \in \Delta$ the fiber is the projective subvariety  corresponding to the projectivization of the $C_5$-cone of $W$ at $q$, that is $\mathbb{P}C_5(W,q)$.     
     This is roughly the way Whitney proved that the $C_5$-cone is an algebraic variety  in \cite[Th. 5.1]{Whi652}.  Finally,  if  $\lambda$ denotes     
      the embedding of $W$ in the diagonal
          \begin{align*}\lambda: W & \hookrightarrow W \times W \\ q & \longmapsto (q,q) \end{align*}
      then $C_5(W)$ is the space obtained by  deprojectivization of the (fibers of)  divisor $D$ and $\psi$ corresponds to the pullback of $e_\Delta$ by $\lambda$:
      
            \[ \xymatrix{  C_5(W) \ar[r] \ar[d]_\psi  & Bl_\Delta (W \times W) \ar[d]^{e_\Delta} \\
                          W \ar[r]_\lambda   &  W \times W}\]   
\end{proof}

\begin{remark}\label{C5facts} 
    
\noindent  (a) The analytic space $C_5(W)$ is of dimension $2 \cdot $dim$(W)$.\\

\noindent  (b) If $Y \subset W$ is a subvariety and $y \in Y$, then $C_5(Y,y) \subset C_5(W,y)$. In particular, when we think of $X \subset \mathbb{C}^n \times \mathbb{C}$ as a family of curves $p:X\rightarrow T$, with a section $\sigma(t):=(z,t)$, $z \in \mathbb{C}^n$, and fibers $X_t:=p^{-1}(t)$, we have that the $C_5$-cone, $C_5(X_t, (z,t))$, of the fiber is contained in the $C_5$-cone, $C_5(X,(z,t))$, of the surface.
   
\end{remark}

We can now state the main result of this section.

\begin{theorem}\label{thgenericproj} Let $p:(X,0)\rightarrow (\mathbb{C},0)$ be a family of generically reduced curves and suppose there is a good representative $p:X\rightarrow T$ with a section $\sigma:T \rightarrow X$ such that $\sigma(T)$ is smooth. \textit{There is a Zariski's open and dense set $\Lambda \subset G(2,n)$ such that for each $H \in \Lambda$, $\pi_H$ is a generic projection for $(X_t,\sigma(t))$ for all $t\in T$}.
\end{theorem}

\begin{proof} 
   Following the proof of Lemma \ref{C5asblowup} we will begin by constructing an analytic space $Z$ such that the fibers of $\Psi:Z  \to X$, $Z((z,t)):=\Psi^{-1}(z,t)$, are cones in $\C^{n+1}$
   lying between the $C_5$-cones of the fiber $X_t$ and the one the surface $X$, {\it i.e.}: 
      \[ C_5(X_{t}, (z,t)) \subset Z((z,t)) \subset C_5(X,(z,t))\] 
	We will then construct a projection whose kernel is transversal to the fibers $Z((z,t))$.      
      
    Consider the fibered product $X \times_p X \subset X \times X$. It is analytic of dimension 3, defined by the equation $t=\tau$ in the coordinate system $(z_1,\ldots, z_n,t,w_1,\ldots,w_n,\tau)$. Now we blow-up the diagonal $\Delta$, and as before:
     
     \[ \xymatrix{    Z  \ar[r] \ar[d]_\Psi  & Bl_\Delta (X \times_p X)  \ar[r] \ar[d]^{\widetilde{e_\Delta}}   & Bl_\Delta (X \times X) \ar[d]^{e_\Delta} \\
                          X \ar[r]_\lambda & X\times_p X \ar@{^{(}->}[r]   &  X \times X}\]
    we obtain the space $Z$  as the deprojectivization of the (fibers of) divisor $\widetilde{e_\Delta}^{-1}(\Delta)$ and $\Psi$ corresponds to the pullback of $\widetilde{e_\Delta}$
    by $\lambda$. Note that  $Z$ is a 3-dimensional space constructed by taking limits of secants in $X$ with the restriction that the two points of $X$ defining the secant must be in the same fiber $X_t$.  
    
    Now,  if $(z,t)$ is a singular point of $X$ then it is a singular point of the fiber $X_{t}$ and we have that  
      \[ \mathrm{dim} \, C_5(X_{t}, (z,t))= 2 \leq \, \mathrm{dim} \, Z((z,t)) \leq\,  \mathrm{dim}\, C_5(X,(z,t)) \leq 4\] 
      
    Since ${\rm dim}\, Z=3$ and ${\rm dim}\, \sigma (T) =1$, then the set $\{ (z,t) \in \sigma(T): \ dim \ Z((z,t))>2 \}$ is either empty or of dimension zero. So we may assume that in a small 
enough representative of $X$ we have $dim \ Z((z,t))=2$ for all $(z,t)\neq 0 \in \sigma(T)$. If we define $Z_T:= \overline{\Psi^{-1}(\sigma(T))\setminus {\Psi^{-1}(0)}}$ then it is an analytic space of dimension at most 3 and the induced map $Z_T \to \sigma(T)$ has fibers of dimension at most 2. In particular, its special fiber, $Z_{T,0}$ has dimension at most two. So there exists a Zarisk's open dense set $\Lambda^{'} \subset G(2,n)$ such that for 
  each $H \in \Lambda$, the corresponding linear projection $\pi_H:\mathbb{C}^n \rightarrow \mathbb{C}^2$ has a kernel, $H$, transversal to the fiber $Z_{T,0}$. Since $Z_T$ is analytic, nearby fibers will also be transversal to $H$. Recalling that $C_5(X_t, \sigma(t))\subset Z_{(T, \sigma(t))}$ we conclude that 

\begin{center}
$C_5(X_t,\sigma(t)) \cap H =\lbrace 0\rbrace \subset \mathbb{C}^n  $
\end{center}
for every sufficiently small $t$. In other words, for $H \in \Lambda^{'}$, the projection $\pi_H$ is a $C_5$-generic for $X_t$, $t \neq 0$. Now, there is a Zariski's open and dense set $\Lambda^{''} \subset G(2,n)$ such that for each $H \in \Lambda^{''}$, the projection $\pi_H$ is $C_5$-generic for $(X_0,0)$ (see \cite[Prop. IV.2]{briancon}). Hence, $\Lambda=\Lambda^{'}\cap \Lambda^{''} \subset G(2,n)$ is a Zariski's open and dense set such that for each $H \in \Lambda$ we have that $\pi_H$ is a $C_5$-generic projection for $(X_t,\sigma(t))$ for all $t$.\end{proof}

In the sequel, we will state some consequences of Theorem \ref{thgenericproj} that are well known to specialists, however the authors did not find a complete written proof of them. 
  
\begin{corollary}\label{corequisaturated} Consider: $p : X \rightarrow T$ as in Theorem \rm \ref{thgenericproj}. \textit{Suppose that $(X,0)$ is a family of reduced curves and consider a projection $\pi \times id: \mathbb{C}^n \times \mathbb{C}\rightarrow \mathbb{C}^2 \times \mathbb{C}$ such that $\pi$ is generic for all $X_t$, $t\in T$. Set $\tilde{X}:=(\pi \times id) (X)$ and consider the family of plane curves $\tilde{p}:\tilde{X}\rightarrow T$. Then the following statements are equivalent:}\\

\noindent \textit{(a) The family of curves $X$ is equisaturated.}\\
\noindent \textit{(b) The projected surface $\tilde{X}$ is a topologically trivial family of plane curves.}\\
\noindent \textit{(c) The family of curves $X$ is bi-Lipschitz equisingular.}
\end{corollary}

\begin{proof}(a) $\Leftrightarrow$ (b) This equivalence is a consequence of Theorem \ref{thgenericproj} and \cite[Théorème IV.8]{briancon}. The proof of (a) $\Rightarrow$ (c) follows by \cite[Théorème 4]{pham2}.\\ 

\noindent (c) $\Rightarrow$ (b) Since $X$ is bi-Lipschitz equisingular, there is a bi-Lipschitz homemomorphism $h:X\rightarrow X_0 \times T$. In particular, the restriction $h_t:X_t\rightarrow X_0 \times \lbrace t \rbrace \simeq X_0$ is a bi-Lipschitz homeomorphism for each $t \in T$. We have that $\pi_t:=\pi \times \lbrace t \rbrace: X_t \rightarrow \tilde{X}_t$ is a bi-Lipschitz homeomorphism (\cite{Tei82}, see also \cite{pichon}, Thm. $5.1$). Hence, the map $\pi_t \circ h_t \circ \pi_t^{-1}: \tilde{X}_t\rightarrow \tilde{X}_0$ is a bi-Lipschitz homeomorphism. By \cite[Theorem $1.1$]{pichon}, the curves $(\tilde{X}_0,0)$ and $(\tilde{X}_t,\sigma(t))$ have the same (embedded) topological type, hence the family $p:\tilde{X}\rightarrow T$ is topologically trivial.\end{proof}

\begin{remark} In the proof of Corollary \rm\ref{corequisaturated} \textit{we made use of} \rm\cite[Théorème IV.8]{briancon}, \textit{which is stated for a deformation of a plane curve $(X_0,0)$. However, if one analyzes the proof of this result, the hypothesis that $(X_0,0)$ is a plane curve is not needed. The authors believe this is a typing error in} \rm\cite{briancon}, \textit{since this result is used in an example where $(X_0,0)$ is not a plane curve (see} \rm\cite[Exemple V.3]{briancon}\textit{)}.
\end{remark}

\begin{corollary}\label{corequibili} If the family of reduced curves $X$ is bi-Lipschitz equisingular, then it is Whitney equisingular.
\end{corollary}

\begin{proof} The proof follows by Corollary \ref{corequisaturated} and \cite[Corollaire IV.9]{briancon}.\end{proof}

\section{The problem of the number of tangents}\label{sec3}

$ \ \ \  $ When $X$ is a topologically trivial family of generically reduced curves, we know by \cite[Theorem 4.4]{jawad} that it admits a normalization in family. So if in addition $(X,0)$ is irreducible, then each curve $(X_t,\sigma(t))$ is also irreducible. Therefore, all the fibers have the same number of tangents, one for each fiber. Throughout this section, we will consider equisingular families of curves with reducible deformation space. 

\subsection{The $s$-invariant}\label{sections}

$ \ \ \  $ In this section we give an answer to Question 1 stated in the introduction. We introduce the $s$-invariant that gives us some information on tangency between two curves. For two germs of plane curves $(\mathscr{C},0)$ and $(\mathscr{C}^{'},0)$, we denote the intersection multiplicity of $(\mathscr{C},0)$ and $(\mathscr{C}^{'},0)$ at $0$ by $i(\mathscr{C},\mathscr{C}^{'},0)$ (see for instance \cite[Sec. $3.2$]{greuel16}).

\begin{definition}\label{defs}
Let $(\mathscr{C},0) \subset (\mathbb{C}^n,0)$ be a germ of reduced curve and denote by $(\mathscr{C}^{i},0)$, $i=1,\cdots, r$, its irreducible components. Let $\pi: (\mathscr{C},0)\rightarrow (\mathbb{C}^2,0)$ be a $C_5$-generic projection of $(\mathscr{C},0)$, then $(\pi(\mathscr{C}),0)$ is a germ of reduced plane curve which we will denote by $(\tilde{\mathscr{C}},0)$ and call $(\tilde{\mathscr{C}}^{i},0)$ its irreducible components. Suppose that $r\geq 2$, then we define the $s$-invariant of $(\mathscr{C},0)$ at the point $0$ as:

\begin{center}
$s(\mathscr{C},0):=  \displaystyle { \sum_{1\leq i<j \leq r}^{}}i(\tilde{\mathscr{C}}^{i},\tilde{\mathscr{C}}^{j},0) $
\end{center}

\noindent If $(\mathscr{C},0)$ is a germ of generically reduced curve, we consider the reduced structure of the $C_5$-generic projection and define $s(\mathscr{C},0)$ in the same way. 
\end{definition}

Note that the $s$-invariant is well defined. If $(\mathscr{C},0)$ is reduced and $\pi,\pi^{'}:(\mathscr{C},0)\rightarrow (\mathbb{C}^2,0)$ are two generic projections of $(\mathscr{C},0)$, then the topological type of $(\pi(\mathscr{C}),0)$ and $(\pi^{'}(\mathscr{C}),0)$ are the same, therefore the intersection multiplicities of the branches of both projections are the same (see \cite{briancon}, Proposition IV.2). 

The invariant is also well defined for generically reduced curves. In fact, since $(\mathscr{C},0)$ is generically reduced and $\pi$ is an isomorphism outside the origin, by \cite[Prop. 1.5]{pellikaan} we have that $( \tilde{\mathscr{C}},0)$ is also generically reduced, considering its structure induced by Fitting ideals.  By definition of the $C_5$-cone, we have that a projection $\pi:\mathbb{C}^n\rightarrow \mathbb{C}^2$ is $C_5$-generic for $(\mathscr{C},0)$ if and only if is $C_5$-generic for the reduced curve $(|\mathscr{C}|,0)$. Let $\pi:\mathbb{C}^n\rightarrow \mathbb{C}^2$ be a $C_5$-generic projection for $(\mathscr{C},0)$ and consider the inclusions $\iota_1:(|\mathscr{C}|,0) \hookrightarrow (\mathscr{C},0)$ and $\iota_2:(|\tilde{\mathscr{C}|},0) \hookrightarrow (\tilde{\mathscr{C}},0)$. Thus, we have that $ \pi_{|_{(\mathscr{C},0)}} \circ \iota_1 = \iota_2 \circ  \pi_{|_{(|\mathscr{C}|,0)}}$, hence $s$ is well defined also in this case.\\

The following two propositions show that the $s$-invariant satisfies the properties we need in order to study tangency.

\begin{proposition}\label{propprinc} Let $(\mathscr{C},0)$ be a germ of curve as in Definition \rm \ref{defs}.  \textit{Then}\\

\noindent \textit{(a) $s(\mathscr{C},0) \geq \displaystyle{\sum_{1\leq i<j \leq r}^{}} \ m(\mathscr{C}^{i},0)\cdot m(\mathscr{C}^{j},0)$ }.\\

\noindent \textit{(b) $s(\mathscr{C},0) = \displaystyle{\sum_{1\leq i<j \leq r}^{}} \ m(\mathscr{C}^{i},0)\cdot m(\mathscr{C}^{j},0)$ if and only if $(\mathscr{C}^{i},0)$ and $(\mathscr{C}^{j},0)$ do not have the same tangent for any $i\neq j$.}
\end{proposition}

\begin{proof}

(a) By the additivity formula of Serre (see \cite[Chapter V]{serre}), for a generically reduced curve $\mathscr{C}$, we have  $m(\mathscr{C},0)=m(|\mathscr{C}|,0)$ and $m(\mathscr{C}^{i},0)=m(|\mathscr{C}|^{i},0)$ for all $i$ (see also \cite[Lemma 4.8]{otoniel}). So, by \cite[Theorem VI.1.6]{pham} we conclude that $m(\mathscr{C}^{i},0)=m(\tilde{\mathscr{C}}^{i},0)$ for all $i$. Statement $(a)$ now follows since $i(\tilde{\mathscr{C}}^{i},\tilde{\mathscr{C}}^{j},0)\geq m(\tilde{\mathscr{C}}^{i},0)\cdot m(\tilde{\mathscr{C}}^{j},0)$ (see for instance \cite[Prop. $3.21$]{greuel16}).\\

\noindent(b) It is not hard to see that a $C_5$-generic projection $\pi:(\mathscr{C},0)\rightarrow (\mathbb{C}^2,0)$ preserves tangency between curves, that is $(\mathscr{C}^{i},0)$ and $(\mathscr{C}^{j},0)$ are tangent if and only if $(\tilde{\mathscr{C}}^{i},0)$ and $(\tilde{\mathscr{C}}^{j},0)$ are tangent. Statement $(b)$ now follows since $i(\tilde{\mathscr{C}}^{i},\tilde{\mathscr{C}}^{j},0)= m(\tilde{\mathscr{C}}^{i},0)\cdot m(\tilde{\mathscr{C}}^{j},0)$ if and only if $(\mathscr{C}^{i},0)$ and $(\mathscr{C}^{j},0)$ are not tangent (see \cite[Prop. $3.21$]{greuel16}).\end{proof}

Next proposition shows that under the hypothesis of topological triviality the invariant $s$ is upper semi-continuous.

\begin{proposition}\label{propprinc2}  Let $p:(X,0) \subset (\mathbb{C}^{n+1},0) \rightarrow (\mathbb{C},0)$ be a family of generically reduced curves and suppose there is a good representative $p:X\rightarrow T$ with a section $\sigma:T \rightarrow X$ such that both $\sigma(T)$ and $X_t \setminus \sigma(t)$ are smooth for $t \in T$. If $X$ is topologically trivial, then $s(X_0,0)\geq s(X_t,\sigma(t)), \ t\in T$.
\end{proposition}

\begin{proof} Let $\tilde{\pi}:\mathbb{C}^n \times \mathbb{C}\rightarrow \mathbb{C}^2 \times \mathbb{C}$ be a projection defined by $\tilde{\pi}(x,t)=(\pi(x),t)$, such that $\pi|_{X_t}: X_t \rightarrow \mathbb{C}^2$ is $C_5$-generic for $X_t$, for all $t\in T$; such a generic projection exists by Theorem \ref{thgenericproj}. Set $\tilde{X}:=\tilde{\pi}(X)$ and consider the Fitting structure for $\tilde{X}$ (see \cite[p. 48]{greuel16}). The map $\tilde{\pi}$ restricted to $X$ is generically one-to-one over its image $\tilde{X}$, however embedded components can appear in $\tilde{X}$ and it can be not reduced. So, we consider the reduced structure for $\tilde{X}$, denoted by $| \tilde{X} |$.

We can see the projection $\tilde{p}:|\tilde{X}|\rightarrow T$ as a family of reduced curves with fibers $\tilde{X}_t:= \tilde{p}^{-1}(t)$, where $\tilde{p}$ is the restriction to $| \tilde{X} |$ of the composition $p \circ \iota$ where $\iota$ is the inclusion of $\mathbb{C}^3$ in $\mathbb{C}^{n+1}$. We can also consider a section $\tilde{\sigma}:T \rightarrow |\tilde{X}|$ defined as $\tilde{\sigma}:= \pi \circ \sigma$. Note that each fiber $\tilde{p}^{-1}(t)$ is the image of a $C_5$-generic projection of $X_t$. Therefore, we have a family of $C_5$-generic projections. 

Let $X=X^1\cup \cdots \cup X^r$ be a decomposition into irreducible components of $X$ and let $p_i:X^i\rightarrow T$ be the restriction of $p$ to $X^i$. Set $X_t^i:={(p^i)}^{-1}(t)$. The hypothesis that $\pi$ is a $C_5$-generic projection for $X_t$ implies that:\\

(1) The family $|\tilde{X}|$ has a decomposition $|\tilde{X}|=|\tilde{X}^1| \cup \cdots \cup |\tilde{X}^r|$ into irreducible components such that $\tilde{X}^i$ is the image of $X^i$ by $\tilde{\pi}$. So, let $\tilde{p}^i:|\tilde{X}^i| \rightarrow T$ be the restriction of $\tilde{p}$ to $|\tilde{X}^i|$. Set $\tilde{X}_t^i:={(\tilde{p}^i)}^{-1}(t)$.

(2) The number of irreducible components of the fibres $X_t^i$ and $\tilde{X}_t^i$ is the same.\\

Since $X$ is topologically trivial, then each $X^i$ is also topologically trivial. 
Hence, by (2) the fibre $\tilde{X}_t^i$ is irreducible for all $t\in T$. So there are analytic irreducible functions $F_1,\cdots,F_r \in \mathbb{C}\lbrace x,y,t \rbrace$ such that $|\tilde{X}^i|=V(F_i)$, and for all $t_0 \in T$, $\tilde{X}^i_{t_0}=V(F_i(x,y,t_0))$ and $F_i(x,y,t_0) \in \mathbb{C}\lbrace x,y \rbrace$ is analytically irreducible.

Set $f_i=F_i(x,y,0)$. Since each $f_i$ is irreducible and $\tilde{X}_0$ is reduced, all $f_1,\cdots,f_r$ are distinct. In particular, $f_1,\cdots,f_r$ does not have any common factor. Note that we can see the function $F_i$ as an unfolding of $f$, thus by \cite[Prop. $3.14$]{greuel16} (see also \cite[Th. $6.4.1$]{jong}) we have that:

\begin{equation}\label{eqproof}
i(f_i,f_j)=\displaystyle { \sum_{x \in \tilde{X}_t^i \cap  \tilde{X}_t^j}^{}}i(\tilde{X}^{i}_t,\tilde{X}^{j}_t,x)
\end{equation}

\noindent where $i(f_i,f_j)$ denotes the intersection multiplicity of the functions $f_i$ and $f_j$. Note that $i(\tilde{X}_0^i, \tilde{X}_0^j,0)=i(f_i,f_j)$ and for each $t\in T$, $\tilde{\sigma}(t) \in \tilde{X}_t^i \cap \tilde{X}_t^j$. Thus, 

\begin{center}
$ \ \ \ \ \ \ \  s(X_0,0)= \displaystyle { \sum_{i<j}^{}}i(\tilde{X}^{i}_0,\tilde{X}^{j}_0,0) \overset{eq(1)}{=} \displaystyle { \sum_{i<j}^{}} \left( \displaystyle { \sum_{x \in \tilde{X}_t^i \cap \tilde{X}_t^j}^{}}i(\tilde{X}^{i}_t,\tilde{X}^{j}_t,x) \right) \geq \displaystyle { \sum_{i<j}^{}} i(\tilde{X}^{i}_t,\tilde{X}^{j}_t,\tilde{\sigma}(t))=s(X_t,\sigma(t)). \ \ \ \ \ \ \ \ \  \qedhere $
\end{center}
\end{proof}

\begin{remark} In general, the invariant $s$ may not be upper semi-continuous. For instance, let $(X,0) \subset (\mathbb{C}^3,0)$ be the surface defined by the zeros of $f(x,y,t)=y\cdot(x^2+y^3+t^2y^2)$. Consider the family $p:(X,0) \rightarrow (\mathbb{C},0)$ where $p$ is the restriction of the canonical projection from $\mathbb{C}^3$ to the last factor. Take a good representative $p:X\rightarrow T$ and consider the section $\sigma:T\rightarrow X$ defined by $\sigma(t)=(0,0,t)$. We have that $X_t \setminus \sigma(t)$ is smooth for all $t \in T$, but $s(X_0,0)=2$ and $s(X_t,\sigma(t))=3$ for $t\neq 0$.\end{remark}

Let $(X,0)$ be a family of generically reduced curves, the following theorem gives us sufficient conditions for the tangent cones $C(X_0,0)$ and $C(X_t,\sigma(t))$ to be homeomorphic.

\begin{theorem}\label{teoprincipal} Let $p:(X,0) \subset (\mathbb{C}^{n+1},0) \rightarrow (\mathbb{C},0)$ be a family of generically reduced curves with $(X,0)$ equidimensional and suppose there is a good representative $p:X\rightarrow T$ with a section $\sigma:T \rightarrow X$, such that both $\sigma(T)$ and $X_t \setminus \sigma(t)$ are smooth for $t \in T$. \textit{If $\mu(X_t,\sigma(t))$, $m(X_t,\sigma(t))$ and $s(X_t,\sigma(t))$ are constant for all $t \in T$, then the tangent cones $C(X_0,\sigma(0))$ and $C(X_t,\sigma(t))$ are homeomorphic. That is, the number of tangents of $(X_t,\sigma(t))$ is constant along $\sigma(T)$.}
\end{theorem}

\begin{proof} We have that $m(X_t,\sigma(t))$ is constant, then the fiber $X_t$ is connected for all $t \in T$ by \cite[Lemma 4.6]{otoniel}. By Theorem \ref{teotoptrivial} we have that $X$ is topologically trivial. Suppose $X=X^{1}\cup X^{2} \cup \cdots \cup X^{r}$ is the decomposition of $X$ into irreducible components, which are all, of dimension two and $X^j \cap X^i = \sigma(T)$ for all $i\neq j$.
 
Since $X$ is topologically trivial, the restriction of $p$ to any component $X^{i}$ and to any finite union of components induces a topologically trivial family of generically reduced curves. We define then, for each pair $X^{i}$ and $X^{j}$, the family $p_{ij}:= p|_{X^{i}\cup X^{j}} : X^{i}\cup X^{j}\rightarrow T$, which is topologically trivial. Furthermore, since the multiplicity of the fibers is
upper  semi-continuous, the multiplicity of $X_t^{i}$ at $\sigma(t)$ is constant for every $i$. On the other hand, since $X$ is topologically trivial,  the semi-continuity property of the $s$-invariant, (Prop. \ref{propprinc2}), implies that $s(X_t,\sigma(t))$ is constant if and only if $s(X_t^{i}\cup X_t^{j},\sigma(t))$ is constant for each pair $(i,j)$, with $i<j$.

Suppose now that the number of tangents of $(X_t,\sigma(t))$ is not constant. Then, for a sufficiently small representative $X$ of $(X,0)$ there are two irreducible components $X^{i}$ and $X^{j}$ of $X$ for which one of the following situations holds:\\

$\bullet$ \textit{Situation (a):} $X_t^{i} \cup X_t^{j}$ has two tangents for $t=0$ and only one tangent for $t \neq 0$. In this case, by Proposition \ref{propprinc} we have that:

\begin{center}
$s(X_0^{i}\cup X_0^{j},0)= m(X_0^{i}, 0)\cdot m(X_0^{j}, 0)$ $ \ \ \ $  and $ \ \ \ $ $s(X_t^{i}\cup X_t^{j},\sigma(t))> m(X_t^{i},\sigma(t))\cdot m(X_t^{j},\sigma(t))$ for $t\neq 0$. 
\end{center}

\noindent Since the multiplicities of $X_t^{i}$ and $X_t^{j}$ at $\sigma(t)$ are constant, it follows that $s(X_0^{i}\cup X_0^{j},0)<s(X_t^{i}\cup X_t^{j},\sigma(t))$, which is a contradiction.\\

$\bullet$ \textit{Situation (b):} $X_t^{i} \cup X_t^{j}$ has one tangent for $t=0$ and two tangents for $t \neq 0$ at $\sigma(t)$. Again, by Proposition \ref{propprinc} we have that

\begin{center}
$s(X_0^{i}\cup X_0^{j},0)> m(X_0^{i}, 0)\cdot m(X_0^{j}, 0)$ $ \ \ \ $  and $ \ \ \ $ $s(X_t^{i}\cup X_t^{j},\sigma(t))= m(X_t^{i},\sigma(t))\cdot m(X_t^{j},\sigma(t))$ for $t\neq 0$. 
\end{center}

\noindent Since the multiplicities of $X_t^{i}$ and $X_t^{j}$ at $\sigma(t)$ are constant, it follows that $s(X_0^{i}\cup X_0^{j},0)>s(X_t^{i}\cup X_t^{j},\sigma(t))$, which is again a contradiction.\end{proof}

By Theorem \ref{teotoptrivial} we have another formulation of Theorem \ref{teoprincipal}:

\begin{corollary}\label{corprinc} Let $p:X\rightarrow T$ be as in Theorem \rm \ref{teoprincipal}.\\

\noindent \textit{(a) If $X$ is topologically trivial and $m(X_t,\sigma(t))$ and $s(X_t,\sigma(t))$ are constant for all $t \in T$, then $C(X_0,\sigma(0))$ and $C(X_t,\sigma(t))$ are homeomorphic.}\\

\noindent \textit{(b) If $X$ is Whitney equisingular and $s(X_t,\sigma(t))$ is constant for all $t \in T$, then $C(X_0,\sigma(0))$ and $C(X_t,\sigma(t))$ are homeomorphic.}

\end{corollary}

\begin{corollary}\label{cor1} Let $p:X\rightarrow T$ be as in Theorem \rm \ref{teoprincipal}. \textit{Suppose that $X$ is an equisaturated family of reduced curves, then the tangent cones $C(X_0,0)$ and $C(X_t,\sigma(t))$ are homeomorphic.}
\end{corollary}

\begin{proof}

Consider the projection $\pi \times id: (\mathbb{C}^n \times \mathbb{C},0)\rightarrow (\mathbb{C}^2 \times \mathbb{C},0)$ defined by $\pi \times id (x,t):=(\pi(x),t)$, where $\pi$ is a $C_5$-generic projection for $X_t$ for all $t$. Denote by $(\tilde{X},0)$ the image of $(X,0)$ by $\pi \times id$, thus $\tilde{p}:(\tilde{X},0)\rightarrow (\mathbb{C},0)$ is a family of plane curves with fiber $\tilde{p}^{-1}(t):=\tilde{X}_t$. Since $(X,0)$ is equisaturated, we have by \cite[Théorème IV.8]{briancon} that $(X,0)$ and $(\tilde{X},0)$ are topologically trivial and hence $\mu(X_t,\sigma(t))$ and $\mu(\tilde{X}_t,\sigma(t))$ are constant. By \cite[Corollaire IV.9]{briancon} we have that $m(X_t,\sigma(t))$ is constant. Consider $(\tilde{X},0)=(\tilde{X}^{1}\cup \cdots \cup \tilde{X}^{r},0)$ the decomposition of $(\tilde{X},0)$ into irreducible components and denote by $(\tilde{X}_t^i,\sigma(t))$ the fiber of $\tilde{p}^{-1}(t)$ restricted to $(\tilde{X}^i,\sigma(t))$. Since $\mu(\tilde{X}_t,\sigma(t))$ is constant and $\tilde{p}:(\tilde{X},0)\rightarrow (\mathbb{C},0)$ is a family of plane curves, $\mu(\tilde{X}_t^i,\sigma(t))$ is constant for all $i=1,\cdots, r$. By Hironaka's formula (\cite[Prop. 4]{hiro}) we have that

\begin{center}
$\mu(\tilde{X}_t,\sigma(t))=\mu(\tilde{X}_t^1,\sigma(t))+\cdots+\mu(\tilde{X}_t^r,\sigma(t))+2s(X_t,\sigma(t))-r+1$,
\end{center} 

\noindent hence, $s(X_t,\sigma(t))$ is constant. The statement now follows by Theorem \ref{teoprincipal}.\end{proof}

\begin{remark} By Corollary \rm\ref{corequibili}, \textit{we can also see Corollary} \rm\ref{cor1} \textit{as a particular case of a result by Sampaio (see} \cite[\textit{Th. $2.2$}]{edson}\textit{).} 
\end{remark}

\begin{corollary} Let $p:X\rightarrow T$ be as in Theorem \rm \ref{teoprincipal}. \textit{If $X$ is a topologically trivial family of reduced plane curves, then the tangent cones $C(X_0,0)$ and $C(X_t,\sigma(t))$ are homeomorphic.}
\end{corollary}

\begin{proof}\let\qed\relax

A topologically trivial family of reduced plane curves is equisaturated, thus the result follows by Corollary \ref{cor1}.\end{proof}

\subsection{Examples and Counterexamples}\label{sec5}

$ \ \ \  $ The purpose of this section is to provide counterexamples for natural questions about the results of the previous sections. First, we present  Proposition \ref{prop1} which will be very useful in the presentation of the examples. 
  
\begin{definition}
Let $p_i:(X^i,0)\rightarrow (\mathbb{C},0)$; $i=1,2$, be a family of generically reduced curves in $(\mathbb{C}^n \times \mathbb{C},0)$, the surface $(X^i,0)$ being defined by the ideal $\mathcal{I}^i \in \mathcal{O}_{n+1}$. Suppose that there is a good representative $p_i:X^i\rightarrow T$ with a section $\sigma:T \rightarrow X^i$, such that $\sigma(T)$ is smooth and $X^i_t \setminus \sigma(t)$ is smooth for all $t \in T$. We say that $(X^1,0)$ and $(X^2,0)$ intersect well if $\sqrt{(\mathcal{I}^1+\mathcal{I}^2)}=(x_1,\cdots,x_n)$ in $\mathcal{O}_{n+1}$, {\it i.e}, in a good representative $Sing(X) = X^1 \cap X^2= \sigma(T) $. 

\end{definition}

Now we want to be able to study equisingularity in family by cutting into irreducible components of the deformation space.

\begin{proposition}\label{prop1}
For $i=1, \cdots, r$, let $p_i:(X^i,0)\rightarrow (\mathbb{C},0)$ be a family of generically reduced curves 
in $(\mathbb{C}^n \times \mathbb{C},0)$, where $(X^i,0)$ is equidimensional and defined by the ideal $\mathcal{I}^i \in \mathcal{O}_{n+1}$. Suppose that $(X^i,0)$ and $(X^j,0)$ intersect well for all $i\neq j$. Let $p: (X,0) :=$ $\displaystyle { \bigcup_{i=1}^{r}}(X^i,0) \rightarrow (\mathbb{C},0)$ whose restriction to each $(X^i,0)$ is $p_i$. Then:\\

\noindent (a) $X$ is topologically trivial if and only if every $X^i$ is topologically trivial. 

\noindent (b) $X$ is Whitney equisingular if and only if every $X^i$ is Whitney equisingular.

\noindent (c) In addition, suppose that $X$ and every $X^{i}$ are Cohen-Macaulay. Then, $X$ is equisaturated if and only if every $X^{i}$ is equisaturated and $s(X_t,\sigma(t))$ is constant.\end{proposition}

\begin{proof} Since the surfaces intersect precisely along the section of the projection to $\mathbb{C}$ that coincides with the singular locus of $(X,0)$, statements $(a)$ and $(b)$ are clear. For $(c)$, as in Theorem \ref{thgenericproj}, consider a projection $\tilde{\pi}:(\mathbb{C}^n \times \mathbb{C},0) \rightarrow (\mathbb{C}^2 \times \mathbb{C},0)$, defined by $\tilde{\pi}(x,t)=(\pi(x),t)$, where $\pi:\mathbb{C}^n\rightarrow \mathbb{C}^2$ is a $C_5$-generic projection for $(X_0,\sigma(0))$. In particular, $\pi$ is a $C_5$-generic projection for $(X_t,\sigma(t))$ and $(X_t^{i},\sigma(t))$ for all $t$ and $i$. 

Denote by $\tilde{X}:=\pi(X)$, thus $\tilde{p}:\tilde{X}\rightarrow T$ is a one parameter flat deformation of the curve $\tilde{X}_0:=\tilde{p}^{-1}(0)$ with a section $\tilde{\sigma}:T \rightarrow \tilde{X}$ in the sense of \cite{briancon}, that is, a curve of singularities of $\tilde{X}$ passing through $0$ can appear outside the set $\tilde{\sigma}(T) \setminus \lbrace 0 \rbrace$. Denote the fibers of $\tilde{p}$ by $\tilde{X}_t:=\tilde{p}^{-1}(t)$. Note that $(\tilde{X}_t,\sigma(t))$ is a $C_5$-generic projection of $(X_t,\sigma(t))$. By Hironaka's formula (\cite[Prop. 4]{hiro}), we have that:

\begin{center}
$\mu(\tilde{X}_t,\sigma(t))=\mu(\tilde{X}_t^{1},\sigma(t))+\cdots+\mu(\tilde{X}_t^{r},\sigma(t))+2s(X_t,\sigma(t))+r-1$.
\end{center}

\noindent Hence, by Corollary \ref{corequisaturated}, $X$ is equisaturated if and only if $\mu(\tilde{X}_t,\sigma(t))$ is constant, if and only if, $\mu(\tilde{X}_t^{1},\sigma(t)),\cdots,$ $\mu(\tilde{X}_t^{r},\sigma(t))$ and $s(X_t,\sigma(t))$ are constant, if and only if, every $X^{i}$ is equisaturated and $s(X_t,\sigma(t))$ is constant.\end{proof}

We will now present examples and counterexamples that show our results are optimal. For the computations, we have made use of the software Singular \cite{singular}. In the following examples, $p$ is the restriction to $(X,0)$ of the canonical projection of $(\mathbb{C}^4,0)$ to the last factor. We also consider a good representative which we also call $p:X\rightarrow T$, and $\sigma: T\rightarrow X$ is the section defined by $\sigma(t)=(0, \cdots ,0,t)$ and we always have that $X_t \setminus \sigma(t)$ and $X \setminus \sigma(T)$ are smooth for all $t\in T$.






First, if $\mu(X_t,\sigma(t))$ is not constant, then $X$ is not topologically trivial by Theorem \ref{teotoptrivial}(a). Thus, the number of branches can be not constant. So, it is natural to expect that the number of tangents is not constant even if $m(X_t,\sigma(t))$ and $s(X_t,\sigma(t))$ are constant.

\begin{example}\label{exemploprinc}\textit{The hypothesis about the constancy of $m(X_t,\sigma(t))$ can not be omitted in Theorem \ref{teoprincipal}}.\\

Consider the germ of surface $(X,0)=(X^1\cup X^2,0)\subset (\mathbb{C}^3 \times \mathbb{C},0)$ where $(X^1,0)$ and $(X^2,0)$ are germs of surfaces parametrized by the maps $n_1:(u,t)\mapsto (tu,u^2,u^3,t)$ and $n_2:(u,t)\mapsto (u,u^4,u^6,t)$, respectively. The surfaces $(X^1,0)$ and $(X^2,0)$ are defined respectively by the following ideals in $\mathcal{O}_4$:

\begin{center}
$\mathcal{I}^1= \langle x^2-t^2y,xy-tz,xz-ty^2,z^2-y^3 \rangle $ $ \ \ \ \ \ $ and $ \ \ \ \ \ $ $\mathcal{I}^2= \langle y-x^4,z-x^6 \rangle $.
\end{center}

\noindent $X^1$ and $X^2$ intersect well. By Example \ref{exemplo2.3}, we have that $X^1$ is topologically trivial. Note that  $X^2$ is the product $X_0^2 \times \C$. Hence, by Proposition \ref{prop1}, $X$ is a topologically trivial family of generically reduced curves, so $\mu(X_t,\sigma(t))$ is constant. We have that $s(X_t,\sigma(t))=2$ for all $t$, on the other hand $m(X_0,0)=3$ and $m(X_t,0)=2$ for $t\neq 0$. The fiber $(X_0,0)$ has two tangents and $(X_t,0)$ has only one tangent for $t\neq 0$.
\end{example}

\begin{example}\label{exemploprinc2} \textit{The hypothesis about the constancy of $s(X_t,\sigma(t))$ can not be omitted in Theorem \ref{teoprincipal}}. \textit{That is, Whitney equisingularity is not a sufficient condition for Question $1$}.\\

Consider the germ of surface $(X,0)=(X^1\cup X^2,0)\subset (\mathbb{C}^3 \times \mathbb{C},0)$ where $(X^1,0)$ and $(X^2,0)$ are respectively parametrized by the maps $n_1:(u,t)\mapsto (u^3,u^4,u^5,t)$ and $n_2:(u,t)\mapsto (u^3,u^4+tu^3,u^7,t)$ and defined by the following ideals in $\mathcal{O}_4$:

\begin{center}
$\mathcal{I}^1= \langle y^2-xz,yz-x^3,z^2-x^2y \rangle $ $  \ \ \ $ and $  \ \ \ $ $\mathcal{I}^2= \langle z-xy+tx^2,y^3-x^4-3txy^2+3t^2x^2y-t^3x^3 \rangle $.
\end{center}

\noindent As in previous example, note that $(X^1,0)$ and $(X^2,0)$ intersect well. We have also that $X^1$ and $X^2$ are topologically trivial families of reduced curves. Thus by Proposition \ref{prop1} we have that $X$ is a topologically trivial family of generically reduced curves, so $\mu(X_t,\sigma(t))$ is constant. We have that $m(X_t,\sigma(t))=6$ for all $t$, $s(X_0,0)=13$ and $s(X_t,\sigma(t))=9$ for $t\neq 0$. The fiber $X_t$ has two tangents at $\sigma(t)$ for $t\neq 0$ and $X_0$ has just one at $0$.
\end{example}


\begin{example}\label{exemploprinc4} \textit{The constancy of $\mu(X_t,\sigma(t))$, $m(X_t,\sigma(t))$ and $s(X_t,\sigma(t))$ in theorem \ref{teoprincipal} does not imply that $(X,0)$ is equisaturated (or equivalently, bi-Lipschitz equisingular)}.\\

Consider the germ of surface $(X,0)=(X^1\cup X^2,0)\subset (\mathbb{C}^3 \times \mathbb{C},0)$ where $(X^1,0)$ and $(X^2,0)$ are respectively parametrized by the maps $n_1:(u,t)\mapsto (u^4,u^6+tu^7,u^9,t)$ and $n_2:(u,t)\mapsto (0,u,0,t)$ and defined by the following ideals in $\mathcal{O}_4$:

\begin{center}
$\mathcal{I}^1= \langle y^2-x^3-2txz-t^2x^2y+t^3yz-t^4x^4, \ z^2-x^3y+txyz-t^2x^5 \rangle $ $ \ $ and $ \ $ $\mathcal{I}^2= \langle x,z \rangle $.
\end{center}

\noindent As in the previous examples, note that $(X^1,0)$ and $(X^2,0)$ intersect well. We have that $X^1$ and $X^2$ are topologically trivial families of reduced curves. Thus by Proposition \ref{prop1}, $X$ is a topologically trivial family of generically reduced curves, so $\mu(X_t,\sigma(t))$ is constant. We have also that $m(X_t,\sigma(t))=5$ and $s(X_t,\sigma(t))=4$ for all $t$, thus by Theorem \ref{teoprincipal} we have that the number of tangents of $(X_t,\sigma(t))$ is constant. On the other hand, note that the plane curve parametrized by the maps $\varphi_t^1(u)=(u^4,u^6+tu^7+u^9)$ and $\varphi_t^2(u)=(0,u)$ is a $C_5$-generic projection of the fibre $X_t$, which we will denote by $\tilde{X}_t$. Note that the topological type of $\tilde{X}_0$ and $\tilde{X}_t$ are distinct for all $t\neq 0$, hence $X$ is not equisaturated (equivalently, bi-Lipschitz equisingular) by \cite[Th. IV.$8$]{briancon}.
\end{example}

\begin{example}\label{exemploprinc3} \textit{The hypothesis about the constancy of the $s$-invariant in the converse of Prop. \ref{prop1} (c) is necessary.}\\

 Consider the same surfaces of the Example \ref{exemploprinc2}. We have that $X^1$ and $X^2$ are actually equisaturated families of reduced curves. We have also that $s(X_t,\sigma(t))$ is not constant, and $(X,0)$ is not equisaturated.

\end{example}

\section{Families of isolated surfaces singularities}\label{sec4}

$ \ \ \  $ In this section we will take a quick look at the case of Whitney equisingular families of isolated surface singularities.\\ 

As in the case of families of curves, we consider a flat map $p:(X,0) \rightarrow (\mathbb{C},0)$, whose fibers $X_t:=p^{-1}(t)$, are isolated surface singularities. We will say that such  a family  is Whitney equisingular, if for  sufficiently small representatives $X$ of $(X,0)$ and $T$ of $(\mathbb{C},0)$, there exists a section $\sigma: T \rightarrow X$ with smooth image, such that $X_t \setminus \{\sigma(t)\}$ is non singular and $X \setminus \sigma(T)$ satisfies Whitney's condition $(b)$ along $\sigma(T)$.
 
We know by Thom-Mather's first isotopy theorem (\cite{Mat70} and \cite{Tho69}) that, under Whitney equisingularity, $X$ is locally topologically trivial along $\sigma(T)$. So, in the spirit of this paper, we can wonder if Whitney equisingularity of a flat family of isolated surface singularities is sufficient to ensure that the tangent cones of the fibers are homeomorphic. The following example shows that the answer is negative.
 
\begin{example}\label{exewhi}
Consider the germ of 3 dimensional variety $(X,0) \subset (\C^5,0)$  given as union of two smooth germs $(X^{1},0)$ and $(X^{2},0)$ defined respectively by the ideals $\mathcal{I}^{1}=\left<z-x^2,y\right>$ and $\mathcal{I}^{2}=\left<y-tx+w^2,z\right>$ in $\C\{x,y,z,w,t\}$.

Note that the singular locus of $(X,0)$ is the intersection of the two smooth components, namely the $t$-axis. Since $t$ is not a zero divisor in ${\mathcal O}_{X,0}=\C\{x,y,z,w,t\}/\left(\mathcal{I}^{1} \cap  \mathcal{I}^{2}\right)$, the projection on the $t$-coordinate $p:(X,0) \to (\C,0)$ is flat. 

Consider representatives $X= X^{1}\cup X^{2}$ of $(X,0)$ and $T$ of $(\C,0)$, and define a section $\sigma: T \rightarrow X$ by $\sigma(t) = (0,0,0,0,t)$. Each component $X^{i}\setminus \sigma(T)$ is Whitney regular along $\sigma(T)$; and since $\sigma(T)$ is precisely the intersection of both components, the family $p: X \rightarrow T$ is Whitney equisingular along $\sigma(T)$ (we use a similar argument as in Proposisiton \ref{prop1} where we assumed that the families of curves intersect well).  
           
For a general member of the family, the tangent cone $C(X_t,\sigma(t))$ consists of two different planes and the tangent cone of the special fiber $C(X_0,\sigma(0))$ is only one plane, so they are not homeomorphic.
\end{example}

       When $(X,0)\subset (\C^n,0)$ is a curve, the lines of the tangent cone at $0$ are precisely the lines obtained as limits at $0$ of tangent lines to $X$. In the general case (dimension bigger than 1) the relation is more complcated. 
       
       We say that a hyperplane $H$ is tangent to $X\subset \C^n$, at a smooth point $x$ if $H$ contains the tangent space to $X$ at $x$. A limit at a point $x_0$, of tangent hyperplanes to $X$ is a hyperplane of $\C^n$ that contains a limit of tangent planes to $X$. 
       
       H. Hironaka proved that the hyperplanes corresponding to the points of the dual variety ${\mathbb{P}\check{C}(X,0)} \subset \check{\mathbb{P}}^{n-1}$ of the projectivised tangent cone are all limits of tangent hyperplanes, (see  \cite[Theorem 1.5]{HL75}). In other words, every tangent hyperlane to the tangent cone is a limit of tangent hyperplanes.   
       
\begin{theorem}(L\^e, Teissier; \cite{LT88}, Cor 2.1.3) If $(X,0) \subset (\C^n,0)$ is an equidimensional analytic germ, there exists a finite family of subcones $\{V_\alpha\}$ of the tangent cone $C(X,0)$, which includes the irreducible components of $C(X,0)$, such that the set of limits of tangent hyperplanes correspond to the union of the projective dual varieties of  the $\mathbb{P} V_\alpha$'s.
\end{theorem} 
          
     The family $\{V_\alpha\}$ is called the aur\'eole of  $(X,0)$, and the $V_\alpha$'s that are not irreducible components of the tangent cone are called  \emph{exceptional cones}. 

	In the particular case of surfaces, the aur\'eole consists of the irreducible components of the projectivised tangent cone plus a finite number of marked points inside them; these points correspond to affine lines called the exceptional tangents. So, in the case of surfaces, the limits of tangent hyperplanes are all tangent hyperplanes to the tangent cone plus all hyperplanes containing an exceptional tangent.
                    
\begin{example} Consider the normal surface singularity $(S,0) \subset (\C^3,0)$ defined by 
      \[x^2 + y^2 +z^3=0\]
      The tangent cone is defined by $x^2+y^2=0$, it is a union of two planes intersecting along the line $(0,0,z)$. It has a reduced structure, so by (\cite{Flo13}, Cor. 8.14), the surface must have an exceptional tangent. 
      
Using SINGULAR we can calculate the Nash fiber over the origin, and see that the set of limits of  tangent planes to $S$ at $0$ is the set of all the planes of $\C^3$ containing the $z$-axis. So this axis is the unique exceptional tangent to $S$ at $0$.
\end{example}
     
   So, again in the spirit of this work, it is natural to ask if the exceptional tangents are preserved in a Whitney regular family of isolated surface singularities.  
  It seems that it is known to specialists that the existence of exceptional tangents is preserved in this context. However we believe it is important to state it properly:

\begin{proposition}\label{propexceptionaltang}	
         Let  $\varphi: (X,0) \to (\C,0)$ be a one-parameter flat family of isolated surface singularities. If it is Whitney equisingular then the special fiber $X_0$ has exceptional tangents at $0$ if and only if the general fiber $X_t$ has  exceptional tangents at $\sigma(t)$ for every $t$ arbitrarily close to the origin.
     \end{proposition}
     
     \begin{proof}
     
Let us call $\Sigma$ the image of a section of $\varphi$;  recall that it contains the singular locus of $X$. We will consider the normal/conormal diagram of $X$ relative to $\Sigma$ as in 
\cite{LT88} :
             
             \[\xymatrix{E_\Sigma C(X)\ar[r]^{\hat{e}_\Sigma}\ar[dd]^{\kappa'_X}\ar[ddr]^\zeta & C(X)\ar[dd]^{\kappa_X} \\
             &  \\
             E_\Sigma X\ar[r]_{e_\Sigma}  & X }\]

Where \begin{itemize}
 \item $X \subset \C^N$
 \item $E_\Sigma X \subset X \times P^{N-2}$ and $e_{\Sigma}$ is the blow-up of $X$ along $\Sigma$.
 \item $C(X)\subset X \times \check{P}^{N-1}$ is the conormal space of $X$ and $\kappa_X$ is the conormal morphism.
 \item $E_\Sigma C(X) \subset X \times P^{N-2} \times \check{P}^{N-1}$,  $\hat{e}_{\Sigma}$ is  the blow-up of the conormal cspace $C(X)$ along the pull back of $\Sigma$, and
  $\kappa^{'}_X$ the pull-back of $\kappa$. Call $\zeta= \kappa \circ e'_{\Sigma}$.
\end{itemize}

Let us call $D$ the exceptional fiber of $\xi$, it has an irreducible decomposition $D= { \displaystyle \bigcup_{\alpha}D_{\alpha}}$. Let us call  $V_{\alpha}:= \kappa'(D_{\alpha})$. 
Note that $e_{\Sigma}^{-1}(\Sigma) = \displaystyle \bigcup_{\alpha}V_{\alpha}$, but some of the $V_{\alpha}$'s are stricly contained in some irreducuble component of $e_{\Sigma}^{-1}(\Sigma)$, still for each $\alpha$ we have a map $e_{\alpha}: V_{\alpha} \rightarrow \Sigma$. 

Since $(X, \Sigma)$ satisfies Whitney's condition $(b)$, we have that all the maps $e_{\alpha}: V_{\alpha} \rightarrow \Sigma$ are surjective with fibers having the same dimension; see \cite[Prop. 2.2.4.2]{LT88}. Moreover, if $H$ is a hyperplane in the ambient space, which is not a limit a tangent hyperplanes to $X$ at a point $\{s\}= H\cap \Sigma$, then the collection $V_{\alpha} \cap H$ describes the limits of tangent hyperplanes to the section $X\cap H$ at the point $s$; see \cite[Thm. 2.3.2]{LT88}.

So let us consider hyperplanes $H_t$ of the form $z_n=t$. In that way the section $H_t\cap X$ is precisely the fiber $X_t$. Note that since $(X, \Sigma)$ satisfies Whytney's condition $(a)$, and the tangent space to $\Sigma$ at any point is the $z_n$-axis, the hyperlane $H_t$ is not tangent to $X$ at $s_t: = (0, \ldots , 0,t)$. So the limits of tangent hyperlanes to $X_t$ at $s_t$ are determined by the sections $V_{\alpha }\cap H_t$. Hence $X_t$ has an exceptional tangent at $s_t$ if and only if there exists an $\alpha$ such that $H_t\cap V_{\alpha}$ is a projective point corresponding to the exceptional tangent. We have seen that by Whitnay regularity, for every $\alpha$, the fibers of $V_{\alpha} \rightarrow Y$ are simultaneously empty or not. 
So the existence of an exceptional tangent for $X_t$ at $\sigma (t)$ does not depend on $t$.     
     \end{proof}

We can still ask ourselves if in this case the number of exceptional tangents is constant. The answer is again negative and the example we present   here is one of the classical examples by Briançon and Speder  studied in their paper \cite{BS75a}. We thank Anne Pichon for pointing it out to us.        

 \begin{example}\label{exefinal}        
          Let $(X,0)\subset (\C^4,0)$ be the family of isolated surface singularities, parametrized by the t-axis, given by the equation:
        \[ z^3 + tx^4z +x^6 + y^6=0 \]
        
          It is a Whitney regular family as it has been proved
by Briançon and Speder in the aforementioned article \cite{BS75a} by calculating the Milnor numbers of generic plane sections of the family and appealing to Teissier's  $\mu^*(X)$  criterion.

Note that for every member $\left(X_t,\sigma(t)\right) \subset (\C^3,0)$ of the family, its tangent cone $C(X_t,\sigma(t))$ is the $xy$-plane, and its projective 
           dual is the point $[0:0:1] \in \check{\mathbb{P}}^2$. Now, for the special fiber $\left(X_0,0\right)$ defined by the equation $z^3 + x^6+y^6=0$, if we set
           homogeneous coordinates $[a:b:c]$ in $\check{\mathbb{P}}^2$ the Nash fiber over the origin is given by the equation:
           \[  a^6 + b^6=\prod_{k=0}^5 \left(a - \exp\left(i \frac{\pi + 2k\pi}{6}\right)b\right)=0\]
           So either $a=b=0$ and we get the point $[0:0:1]$ (the projective dual of the tangent cone), or $a\neq 0$ and by setting $b=1$ we get  the points
            \[  \left[ \exp\left(i \frac{\pi + 2k\pi}{6}\right):1:c\right] \rightsquigarrow \exp\left(i \frac{\pi + 2k\pi}{6}\right)x + y + cz=0 \]   
           corresponding to 6 pencils consisting of all planes of $\C^3$ that contain the line
           \[ \exp\left(i \frac{\pi + 2k\pi}{6}\right)x + y = 0; \hspace{0.5in} z=0\]
           This means the special fiber has 6 exceptional tangents.
           
            However, for $t \neq 0$ in an arbirtrarily small neighborhood of the origin we have that the Nash fiber of $X_t$ over the singular point is
            given by the equation:
            \[ a^{12} + 2a^6b^6+ \left( 1 + \frac{4}{27}t^3\right)b^{12}=\prod_{k=1}^{12}(a-\alpha_kb)=0\] 
           and reasoning in the same way as above, we get that the general fiber $(X_t,\sigma(t))$ has 12 exceptional tangents.   

\end{example}

\begin{flushleft}
\textbf{Funding}
\end{flushleft}

This work was supported by \textit{Consejo Nacional de Ciencia y Tecnología (Conacyt)} [221635 to A.G.F., 282937 to O.N.S. and J.S.]; \textit{Fondo Institucional de Fomento Regional para el Desarrollo Científico, Tecnológico y de Innovación (FORDECYT)} [265667 to O.N.S.]; and \textit{Programa de Apoyo a Proyectos de Investigación e Innovación Tecnológica (of the National Autonomous University of Mexico) (PAPIIT)} [113817 to O.N.S and J.S].

\begin{flushleft}
\textbf{Acknowledgement}
\end{flushleft}

The authors would like to thank Professor Bernard Teissier for many helpful conversations on this work and Anne Pichon for the communication of Example \ref{exefinal}.

\end{document}